\documentclass[a4paper]{article}
\usepackage{amsthm,amssymb,amsmath,enumerate,graphicx,epsf}

\newcommand{\COLORON}{0}
\newcommand{\NOTESON}{0}
\newcommand{\Debug}{0}
\usepackage{bbold}
\usepackage[usenames]{color} 
\usepackage{amsthm,amssymb,amsmath,bbm,enumerate,graphicx,epsf,stmaryrd}
\usepackage[bookmarks, colorlinks=false, breaklinks=true]{hyperref} 

\usepackage{authblk}

\newcommand{\comment}[1]{}
\newcommand{\COMMENT}[1]{}

\definecolor{darkgray}{rgb}{0.3,0.3,0.3}
\newcommand{\defi}[1]{{\color{darkgray}\emph{#1}}}

\comment{
	\begin{lemma}\label{}	
\end{lemma}
\begin{proof}

\end{proof}

\begin{theorem}\label{}
\end{theorem} 
\begin{proof} 	

\end{proof}

}

\newtheorem{proposition}{Proposition}[section]

\newtheorem{theorem}[proposition]{Theorem}
\newtheorem{corollary}[proposition]{Corollary}

\newtheorem{lemma}[proposition]{Lemma}

\newtheorem{conjecture}{{Conjecture}}[section]

\newtheorem{problem}[conjecture]{{Problem}}

\newtheorem{question}[conjecture]{{Question}}

\newtheorem{examp}[proposition]{Example}

\theoremstyle{definition}
\newtheorem{rem}[proposition]{Remark}

\newcommand{\FIG}{0}

\ifnum \NOTESON = 1 \newcommand{\note}[1]{ 

\hspace*{-30pt}
	{\color{blue}  NOTE: \color{Turquoise}{\small  \tt \begin{minipage}[c]{1.1\textwidth}  #1 \end{minipage} \ignorespacesafterend }} 
	
	}
\else \newcommand{\note}[1]{} \fi

\newcommand{\afsubm}[1]{ \ifnum \Debug = 1 {\mymargin{#1}}
\fi} 

\ifnum \Debug = 1 
\else  \fi

\ifnum \FIG = 1 
\else  \fi

\ifnum \FIG = 1 
\else  \fi

\ifnum \Debug = 1 \usepackage[notref,notcite]{showkeys}
\fi

\ifnum \COLORON = 0 \renewcommand{\color}[1]{}
\fi

\newcommand{\N}{\ensuremath{\mathbb N}}
\newcommand{\R}{\ensuremath{\mathbb R}}

\newcommand{\Z}{\ensuremath{\mathbb Z}}

\newcommand{\cv}{\ensuremath{\mathcal V}}

\newcommand{\sm}{\backslash}

\makeatletter
\DeclareRobustCommand{\cev}[1]{%
  \mathpalette\do@cev{#1}%
}
\newcommand{\do@cev}[2]{%
  \fix@cev{#1}{+}%
  \reflectbox{$\m@th#1\vec{\reflectbox{$\fix@cev{#1}{-}\m@th#1#2\fix@cev{#1}{+}$}}$}%
  \fix@cev{#1}{-}%
}
\newcommand{\fix@cev}[2]{%
  \ifx#1\displaystyle
    \mkern#23mu
  \else
    \ifx#1\textstyle
      \mkern#23mu
    \else
      \ifx#1\scriptstyle
        \mkern#22mu
      \else
        \mkern#22mu
      \fi
    \fi
  \fi
}

\makeatother

\newcommand{\nin}{\ensuremath{{n\in\N}}}

\newcommand{\pth}[2]{\ensuremath{#1}\text{--}\ensuremath{#2}~path}

\newcommand{\seq}[1]{\ensuremath{(#1_n)_{n\in\N}}} 

\newcommand{\g}{\ensuremath{G\ }}
\newcommand{\G}{\ensuremath{G}}

\newcommand{\Cg}{Cayley graph}

\newcommand{\vt}{vertex-transitive}

\newcommand{\Lr}[1]{Lemma~\ref{#1}}

\newcommand{\Tr}[1]{Theorem~\ref{#1}}

\newcommand{\Sr}[1]{Section~\ref{#1}}

\newcommand{\Prr}[1]{Pro\-position~\ref{#1}}
\newcommand{\Prb}[1]{Problem~\ref{#1}}
\newcommand{\Cr}[1]{Corollary~\ref{#1}}
\newcommand{\Cnr}[1]{Con\-jecture~\ref{#1}}

\newcommand{\lf}{locally finite}

\renewcommand{\iff}{if and only if}
\newcommand{\fe}{for every}

\newcommand{\st}{such that}

\newcommand{\ti}{there is}

\newcommand{\obda}{without loss of generality}

\newcommand{\mymargin}[1]{
 \ifnum \Debug = 1
  \marginpar{%
    \begin{minipage}{\marginparwidth}\small%
      \begin{flushleft}%
        {\color{blue}#1}%
      \end{flushleft}%
   \end{minipage}%
  }%
 \fi
}%

\newcommand{\mySection}[2]{}

\usepackage{mathtools,enumitem}
\usepackage[percent]{overpic}

\usepackage{authblk}
\usepackage[greek,english]{babel}  
\usepackage{hyperref}
\RequirePackage{latexsym} \RequirePackage{amsthm}
\RequirePackage{amsmath}

\usepackage{enumerate}
\RequirePackage{amssymb} \RequirePackage{makeidx}
\usepackage{dsfont}
\usepackage{mathrsfs}

\newcommand{\shm}{\ensuremath{<_{\rm{sh}}}}
\newcommand{\divm}{\ensuremath{<_{\rm{div}}}}
\newcommand{\coem}{\ensuremath{<_{\rm{c}}}}

\newcommand{\Aut}{\mathrm{Aut}}

\def\td{tree-decom\-po\-si\-tion}
\def\qt{quasi-tran\-si\-tive}

\def\qi{quasi-iso\-metric}
\def\lf{locally finite}

\usepackage{authblk}

\begin{document}

\title{A full Halin grid theorem}

\author[1]{Agelos Georgakopoulos\thanks{Supported by  EPSRC grants EP/V048821/1, EP/V009044/1, and EP/Y004302/1.}}
\affil[1]{  {Mathematics Institute}\\
 {University of Warwick}\\
  {CV4 7AL, UK}}
\author[2]{Matthias Hamann}
\affil[2]{  {Department of Mathematics}\\
 {University of Hamburg}\\
  {Germany}}

\date{\today}
\maketitle

\begin{abstract}
Halin's well-known grid theorem states that a graph $G$ with a thick end must contain a subdivision of the hexagonal half-grid. We obtain the following strengthening when $G$ is vertex-transitive and locally finite. Either $G$ is quasi-isometric to a tree (and therefore has no thick end), or it contains a subdivision of the full hexagonal grid.
\end{abstract}
\maketitle

\section{Introduction}

Let $\mathbb{H}$ denote the hexagonal lattice (aka.\ honeycomb lattice). Halin's grid theorem asserts that if a  graph \g contains an infinite family of 1-way infinite paths no two of which can be separated by a finite vertex-set, then \g contains a subdivision of the \defi{half-grid}, i.e.\ the intersection of $\mathbb{H}$ with the upper half-space. The aim of this paper is to show that if \g is \defi{\qt}, i.e.\ it has finitely many orbits of vertices under the action of its automorphism group, then we can improve Halin's theorem to obtain a subdivision of the full grid $\mathbb{H}$. Combined with known results about \qt\ graphs, our result can be summarized as

\begin{theorem} \label{full Halin VT general intro} 
Let \g be a \lf, \qt\ graph that is not \qi\ to a tree. Then \g contains a subdivision of $\mathbb{H}$.
\end{theorem}

When \g is a \Cg\ of a group $\Gamma$, it is well-known that it satisfies the first sentence of \Tr{full Halin VT general intro} \iff\ $\Gamma$ is not virtually-free (\cite{GhHaSur}). Another consequence is that every  \lf, \qt\ graph~$G$ that contains a subdivision of the half-grid must in fact contain a subdivision of the full grid $\mathbb{H}$, see Corollary~\ref{half grid means full grid}.

\medskip
Halin's theorem, as well as our \Tr{full Halin VT general intro}, seem too weak to have any group-theoretic consequences because the grids they provide can have arbitrary distortion compared to the geometry of the host graph \G. In \Sr{sec div} we propose a strengthening of these theorems where we require the (half or full) grid $H$ to respect some of the geometry of \G: we require that  for any two rays of $H$ that diverge, their images in \g also diverge (see \Sr{embs} for definitions). We prove that for finitely generated \Cg s the property of containing such a diverging subdivision of $\mathbb{H}$ is invariant under the choice of the generating set (\Prr{qi preserve div}), and pose related questions.

\medskip

A well-known conjecture of Thomas \cite{ThoWel} postulates that the countable graphs are well-quasi-ordered under the minor relation. Our results suggest that the restriction of Thomas' conjecture to vertex-transitive graphs may be within reach. Indeed, combining \Tr{full Halin VT general intro} with a theorem of Thomassen we will deduce that every \lf, 1-ended, vertex-transitive graph is a minor-twin\footnote{We call two graphs $G,H$ \defi{minor-twins}, if both $G<H$ and $H<G$ hold.} of either $\mathbb{H}$ or the cubic grid $\Z^3$ (\Prr{prop just two}). We discuss this in \Sr{sec wqo}. 

\medskip

A different recent characterisation of being quasi-isometric to a tree is as follows. A \lf\ \qt\ graph $G$ is \qi\ to a tree if and only if for every \lf\ \qt\ graph $H$ \qi\ to~$G$, there exists a finite graph that is not a minor of~$H$~\cite{H-minorExcl}. This generalizes a similar characterization  of virtually free groups by Khukhro~\cite{Khukhro2023}.
Combining this result with our \Tr{full Halin VT general intro} yields that a \lf\ \qt\ graph $G$ has a subdivision of  $\mathbb H$  if and only if \g is \qi\ to a \lf\ \qt\ graph $H$ which has each finite graph as a minor.

\section{Preliminaries} \label{prel}

\subsection{Graphs} 
We call a graph \g \emph{\lf} if every vertex has finite degree.
A \emph{ray} in~$G$ is a one-way infinite path and a \emph{double ray} is a two-way infinite path.
A \emph{comb} is a union of a ray~$R$ with infinitely many disjoint (finite) paths starting at~$R$ and otherwise disjoint from~$R$.
The end vertices of those paths are the \emph{teeth} of the comb.
Two rays in~$G$ are \emph{equivalent} if there are infinitely many pairwise disjoint paths between them.
This is an equivalence relation whose equivalence classes are the \emph{ends} of~$G$.
An end is \emph{thick} if it contains infinitely many pairwise disjoint rays.

We say that two rays $R,L$ in \g \defi{diverge}, if
\fe\ $n\in \N$ they have tails $R' \subseteq R, L' \subseteq L$ satisfying $d(R',L')>n$.

A \defi{plane graph} is a graph endowed with an embedding into the plane $\R^2$. 
A plane graph has \emph{bounded co-degree} if all its face-boundaries have bounded lengths.

\subsection{Embeddings} \label{embs}
 
Let $G$ and~$H$ be graphs and let $\gamma\geq 1$ and $c\geq 0$.
A \emph{$(\gamma,c)$-\qi\ embedding} of~$H$ into~$G$ is a map $f\colon V(H)\to V(G)$ such that the following holds for all $u,v\in V(H)$:
\[
\frac{1}{\gamma}d_H(u,v)-c\leq d_G(f(u),f(v))\leq\gamma d_H(u,v)+c.
\]
If additionally also $d_G(w, f(V(H)))\leq c$ holds for all $w\in V(G)$, then it is a \emph{$(\gamma,c)$-quasi-isometry}.
If the precise constants $\gamma$ and~$c$ do not matter, we simply drop them from the names.

A map $f\colon V(H)\to V(G)$ is a \emph{coarse embedding}, if there exists functions $\rho_-\colon [0,\infty)$ and $\rho_+\colon [0,\infty)$ such that $\rho_-(a)\to\infty$ for $a\to\infty$ and
\[
\rho_-(d_H(u,v))\leq d_G(f(u),f(v))\leq \rho_+(d_H(u,v))
\]
for all $u,v\in V(H)$.

More generally, we say that a map $f\colon X\to Y$ between metric spaces $(X,d_X), (Y,d_Y)$ is \emph{diverging}, if  \fe\ two sequences \seq{x},\seq{y}\ of points of $X$ \st\ $d_X(x_i,y_i) \to \infty$, we have $d_Y(f(x_i),f(y_i)) \to \infty$. We hereby allow \defi{extended metrics} that are allowed to take the value $d_X(x,y)=\infty$, for example when $X$ is a disconnected graph endowed with its \defi{graph distance}, i.e.\ $d_X(x,y)$ is the smallest length of a \pth{x}{y}, or $\infty$ if no such path exists.

A \defi{subdivision} of $H$ in $G$ is a topological embedding $f\colon H \to G$ of the corresponding 1-complexes such that $f(V(H))\subseteq V(G)$. In other words, $f$ maps each vertex of $H$ to one of \G, and it maps each edge $xy$ of $H$ bijectively to an \pth{x}{y} so that the pre-image of each point of $G$ is either empty, or a vertex of $H$, or a unique interior point of an edge of $H$.

Our notion of \defi{diverging subdivision} is obtained by combining the last two definitions.
 
\section{Proof of the main result}

In this section, we will prove our main theorem, Theorem~\ref{full Halin VT general}.
Its proof will be divided into three steps.
First, we will prove that every plane one-ended graph with bounded degree and bounded co-degree and a thick end contains  a subdivision of~$\mathbb H$ (Lemma~\ref{full H planar}).
Then we will prove that every \lf, \qt, plane graph with a thick end contains a subdivision of~$\mathbb H$ (Theorem~\ref{full H planar 2}). Finally, we extend this to the non-planar case and prove our main theorem.

\subsection{The planar 1-ended case}

In order to find a subdivision of~$\mathbb H$, we need a result from~\cite{GeoPapMin} that ensures the existence of a \defi{diverging pair of rays}, i.e.\ a diverging subdivision of the disjoint union of two rays in a graph \G:

\begin{lemma}[{\cite[Theorem 8.16]{GeoPapMin}}] \label{diverging rays}
Let \g be a bounded degree graph that has an infinite set of pairwise disjoint rays. Then $G$ contains a diverging pair of rays.
\end{lemma}

The bounded-degree condition in \Lr{diverging rays}  is necessary as shown by the following example. Let \g be the graph obtained from $\mathbb{H}$ by introducing, \fe\ $\nin$, a new vertex $v_n$ and joining $v_n$ by an edge to each vertex of $\mathbb{H}$ that is at distance $n$ from a fixed root. Notice that \g is quasi-isometric to an 1-way infinite path.

\begin{lemma} \label{full H planar}
Let \g be a plane graph with bounded degrees and co-degrees. Assume moreover that \g has exactly one end, and this end is thick. Then $G$ contains a 
subdivision of $\mathbb{H}$.
\end{lemma}
\begin{proof}
By \Lr{diverging rays}, \g has a pair of diverging rays $R_+,R_-$. We may assume \obda\ that $R:= R_+ \cup R_-$ is a double ray after finite modifications.

It is known that every planar, 1-ended graph admits an embedding in $\R^2$ without accumulation points of vertices \cite[Lemma 12]{ThomassenRichter}, so we will assume from now on that \g is embedded that way. Thus $R$ separates $\R^2$ into two unbounded components $A,B$. 

Let $A_+^1$ be the subgraph of \g spanned by the face-boundaries incident with $R_+$ and contained in $A$. Note that we also put face-boundaries into $A^1_+$ if they share only one vertex with~$R$. Since $R_+,R_-$ are diverging, and the face-boundaries of \g have bounded lengths, it is easy to see that $A_+^1 \sm R$ contains an infinite subgraph. Applying the comb lemma \cite[Lemma 8.2.2]{DiestelBook05} to that subgraph yields an infinite comb $C_+^1$ with teeth in $R_+$. The same argument yields an infinite comb $C_-^1$ with teeth in $R_-$. Notice that the spines $R_+^1, R_-^1$ of $C_+^1, C_-^1$ diverge, because they are at bounded distance from  $R_+, R_-$, respectively. Thus we may assume \obda\ that $R^1:= R_+^1 \cup R_-^1$ is a double ray after finite modifications, and that the teeth of $C_+^1$ and  $C_-^1$ are disjoint.

We repeat the above construction with $R^0:=R$ replaced by $R^1$, and $A$ replaced by the component of $\R^2 \sm R^1$ disjoint from $R$, to obtain an infinite sequence $R^1,R^2,\ldots$ of double rays with pairwise diverging tails, each of those tails sending infinitely many disjoint paths to the previous double ray. We also repeat on the other side $B$ of $R$, to obtain an infinite sequence $R^{-1},R^{-2},\ldots$ of double rays with the same properties. Removing some of the paths between the $R^i$ we obtain a subdivision of $\mathbb{H}$ containing all $R^i, i \in \Z$.
\end{proof}

\subsection{The planar multi-ended case}
In order to handle the planar multi-ended case we will make use of a result from~\cite{CanonicalTTD} that allows us to split multi-ended graphs into smaller building blocks.
This result is formulated in terms of tree decompositions, and we now recall the relevant terminology.

\medskip
Let $G$ be a graph.
A \emph{\td} of~$G$ is a pair $(T,\cv)$ of a tree~$T$ and a family $\cv=(V_t)_{t\in V(T)}$ of vertex sets of~$G$, one for every node of~$T$, such that the following hold:
\begin{enumerate}
\item[(T1)] $V(G)=\bigcup_{t\in V(T)}V_t$;
\item[(T2)] for every $e\in E(G)$ there exists $t\in V(T)$ with $e\subseteq V_t$;
\item[(T3)] $V_{t_1}\cap V_{t_3}\subseteq V_{t_2}$ for all $t_1,t_2,t_3\in V(T)$, where $t_2$ lies on the $t_1$-$t_3$ path in~$T$.
\end{enumerate}
We call $T$ the \emph{decomposition tree} and the elements of~$\cv$ the \emph{parts} of the \td.
The sets $V_{t_1}\cap V_{t_2}$ with $t_1t_2\in E(T)$ are the \emph{adhesion sets} of the \td\ and the \emph{adhesion} of the \td\ is the supremum of the sizes of the adhesion sets.
We denote by $G[V_t]$ the subgraph of~$G$ induced by~$V_t$.
The \emph{torso} of~$V_t$ is $G[V_t]$ together with all (possibly new) edges $xy$ if $x$ and~$y$ lie in a common adhesion set in~$V_t$.

A \emph{separation} of~$G$ is a pair $(A,B)$ of vertex sets such that $A\cup B=V(G)$ and such that there is no edge from $A\sm B$ to $B\sm A$.
Its \emph{order} is the size of $A\cap B$.
A separation $(A,B)$ is \emph{tight} if there are components $C_A, C_B$ of $A,B$, respectively, such that $N(C_A)=A\cap B=N(C_B)$.
The separations \emph{induced by} a \td\ $(T,\cv)$ are those of the form
\[
\left(\bigcup_{t\in V(T_1)}V_t,\bigcup_{t\in V(T_2)}V_t\right),
\]
where $T_1$ and $T_2$ are the two components of $T-e$ for some edge $e$ of~$T$.

A separation \emph{distinguishes} a pair of ends \emph{efficiently} if it separates them and no separation of smaller order does the same.
A \td\ $(T,\cv)$ \emph{distinguishes} all ends \emph{efficiently} if for every pair of ends one of the separations induced by $(T,\cv)$ distinguishes those two ends efficiently.

Let $(T,\cv)$ be a \td.
If a ray $R$ has infinitely many vertices in a part $V_t$, then all rays equivalent to~$R$ must have infinitely many vertices of the same part and we say that the end that contains~$R$ \emph{lives in}~$t$.
If $R$ does not live in any part, then there must exist a ray $R_T$ in~$T$ such that every node on~$R_T$ contains a vertex of~$R$.
Note that $R_T$ is uniquely determined up to equivalence.
Furthermore, any ray equivalent to~$R$ must define a ray of~$T$ equivalent to~$R_T$.
We then say that the end that contains~$R$ \emph{lives in} the end of~$T$ that contains~$R_T$.
We note that if $(T,\cv)$ has finite adhesion $k\in\N$, then every end of~$G$ that lives in an end of~$T$ cannot contain more than $k$ disjoint rays and hence cannot be thick.

If the automorphism group $\Aut(G)$ of \g induces an action on the decomposition tree of a \td\ via the parts, then we call the \td\ \emph{$\Aut(G)$-invariant}. We will use the following result that provides such a decomposition for the kind of graph we are interested in:

\begin{theorem}[{\cite[Theorem 7.3]{CanonicalTTD}}] \label{CHM}
Let \g be a \lf\ graph and let $k\in \mathbb N$.
If every two ends of~\g can be separated by at most $k$ vertices, then there exists an $\Aut(G)$-invariant \td\ of \g that distinguishes its ends efficiently.\footnote{\cite[Theorem 7.3]{CanonicalTTD} is stated in terms of \defi{profiles}, which is a generalisation of ends; see \cite[\S 6]{CanonicalTTD} for details.}
\end{theorem}

In order to apply the 1-ended case to the parts that are given by Theorem~\ref{CHM}, we must ensure that they are \qt.
This will be done by the following result by Esperet, Giocanti and Legrand-Duchesne \cite{EGL-QuasiTransitiveGraphsAvoidingMinor}.

\begin{lemma}[{\cite[Lemma 3.12]{EGL-QuasiTransitiveGraphsAvoidingMinor}}] \label{EGL 3.12}
Let \g be a \qt\ \lf\ graph and let $k\in\mathbb N$.
Let $(T, \cv)$, with $\cv = (V_t)_{t\in V(T)}$, be an $\Aut(G)$-invariant \td\ of~\g whose induced separations are tight and have order at most~$k$.
Then, for every $t\in V(T)$, the stabiliser of~$t$ induces
a \qt\ action on $G[V_t]$.
\end{lemma}

We are now ready for the second step of our proof:

\begin{theorem} \label{full H planar 2}
Let \g be a \lf, \qt, plane graph that has a thick end. Then $G$ contains a subdivision of $\mathbb{H}$.
\end{theorem}

\begin{proof}
By a theorem of Dunwoody~\cite{DunPla}  (see also \cite[Theorem 8.2]{H-planarTrans}) \g is \defi{accessible}, that is, there exists $k\in\mathbb N$ such that every two ends can be separated by at most $k$ vertices.
Thus Theorem~\ref{CHM} implies the existence of an $\Aut(G)$-invariant \td\ $(T,\cv)$ that distinguishes the ends efficiently.
In particular,  the adhesion of $(T,\cv)$ is at most~$k$.
Since $G$ has a thick end, there must be a vertex $t\in V(T)$ such that some thick end of~\g lives in~$t$.
As $(T,\cv)$ distinguishes all ends of~$G$ efficiently, no other end of~$H$ lives in~$t$.
The separations induced by $(T,\cv)$ have order at most~$k$ and are tight by the properties of $(T,\cv)$.
Thus Lemma~\ref{EGL 3.12} implies that $G[V_t]$ is \qt.

Let $H_t$ be obtained from $G[V_t]$ by adding all shortest paths between vertices that lie in a common adhesion set.
Then $H_t$ is connected and, since we did not destroy any symmetry and added only finitely many orbits of vertices under the stabiliser of~$t$, it is \qt, too.
It is straight-forward to check that $H_t$ has a unique thick end.
By removing some orbits of vertices, if necessary, we may assume that $H_t$ is 2-connected.
A result of Kr\"on \cite[Theorem 8 (1)]{KroInf} says that such an $H_t$ has bounded co-degree.
We can thus apply Lemma~\ref{full H planar} and obtain a subdivision of~$\mathbb H$ in~$H_t$, which is also a subdivision of~$\mathbb H$ in~$G$.
\end{proof}

\subsection{The non-planar case}

For the non-planar case we will apply another result of Esperet, Giocanti \& Legrand-Duchesne \cite{EGL-QuasiTransitiveGraphsAvoidingMinor}, that allow us to find an 1-ended quasi-transitive minor $H$ in a quasi-transitive graph \G, where $H$ is planar if \g has no $K_\infty$ minor.

\begin{theorem}[{\cite[Theorem 4.3]{EGL-QuasiTransitiveGraphsAvoidingMinor}}]\label{thm_EGL}
Let $G$ be a \lf\ graph excluding $K_\infty$ as a minor and let $\Gamma$ be a group with a \qt\ action on~$G$.
Then there is an integer $k$ such that $G$ admits a $\Gamma$-invariant tree-decomposition $(T, \cv)$ with $\cv = (V_t)_{t\in V(T)}$ and of adhesion at most~$3$ such that for every $t\in V(T)$ the torso of $V_t$ is a $\Gamma_t$-\qt\ minor of~$G$ which is either planar or has treewidth at most~$k$.
\end{theorem}

We can now complete the proof of our main result \Tr{full Halin VT general intro}, which we restate for convenience: 

\begin{theorem} \label{full Halin VT general} 
Let \g be a \lf, \qt\ graph that is not \qi\ to a tree. Then \g contains a subdivision of $\mathbb{H}$.
\end{theorem}

\begin{proof} 
Let us suppose that \g does not contain a subdivision of~$\mathbb H$. Since $\mathbb H$ has maximum degree 3, \g does not contain $\mathbb H$ as a minor either.
In particular, \g does not contain $K_\infty$ as a minor.
Thus Theorem~\ref{thm_EGL} implies that \g has an $\Aut(G)$-invariant \td\ $(T,\cv)$ of adhesion at most~$3$ such that for every $t\in V(T)$ the stabiliser of~$t$ acts \qt ly on the torso of~$V_t$ and this torso either is planar or has treewidth at most $k\in\mathbb N$.

It is known that every \qt\ \lf\ graph without any thick end is \qi\ to a tree  \cite[Theorem 5.5]{KroMolQua}.
Thus, $G$ has a thick end.
Since the adhesion of $(T,\cv)$ is at most~$3$, no thick end of~\g lives in an end of~$T$.
Thus, there exists a vertex $t$ of~$T$ such that some thick end $\omega$ of~\g lives in~$V_t$.
This means that the torso $H_t$ of~$t$ cannot have finite treewidth, and hence is planar.
Rays in~$\omega$ define rays in the torso by restriction.
Easily, equivalence of rays is preserved too.
Furthermore, infinitely many pairwise disjoint rays in~$\omega$ also define infinitely many (equivalent) pairwise disjoint rays in~$H_t$.
Thus $H_t$ has a thick end.
Theorem~\ref{full H planar 2} now implies the existence of a subdivision of~$\mathbb H$ in~$H_t$.
This is also a subdivision of~$\mathbb H$ in~$G$, which contradicts our choice of~$G$.
\end{proof}

Since one-ended quasi-transitive graphs have a thick end (\cite[Proposition 5.6]{Thomassen-Hadwiger}), they cannot be quasi-isometric to a tree and hence Theorem~\ref{full Halin VT general} implies

\begin{corollary} \label{full Halin VT}
Let \g be an 1-ended, locally finite, \vt\ graph. Then \g contains a subdivision of $\mathbb{H}$.
\end{corollary}

Another corollary is that the existence of a half-grid subdivision implies the existence of a $\mathbb H$ subdivision:

\begin{corollary}\label{half grid means full grid}
Every  \lf, \qt\ graph that contains a subdivision of the half-grid contains a subdivision of the full grid $\mathbb{H}$.
\end{corollary}

This follows immediately by combining \Tr{full Halin VT general} with the following: 

\begin{lemma}
Let \g be a bounded degree graph which is quasi-isometric to a tree. Then \g contains no subdivision of the half-grid.
\end{lemma}

\begin{proof}
Let  $f$ be a $(\gamma,c)$-quasi-isometry from~$G$ to a tree~$T$. 
Since \g has bounded degrees, for every $r\in\N$ there exists $b(r)\in\N$ such that every vertex set of \g of diameter at most~$r$ contains at most $b(r)$ vertices.
In particular, for every vertex $t$ of~$T$, there is only a bounded number of vertices of~$G$ mapped to~$t$. Thus there exists an upper bound on the number of disjoint rays that contain vertices mapped to a common vertex of~$T$.

Suppose \g contains a  subdivision of the half-grid. Recall that no two rays of the half-grid can be separated by a finite vertex set. This implies that there exists a ray $R$ in~$T$ such that the $f$-images of any ray $P$ in the half grid intersect the $(\gamma+c)$-ball around a tail of~$R$ infinitely often. Moreover, all but finitely many vertices of~$R$ will have the $f$-image of a vertex of~$P$ in distance at most $\gamma+c$.

Since each ball of radius $\gamma+c$ around any vertex of~$R$ contains the images of only a bounded number of vertices of \G, the above remarks contradict the fact that the half grid contains infinitely many pairwise disjoint rays.
\end{proof}

\section{Diverging subdivisions of $\mathbb H$} \label{sec div}

Halin's theorem provides a subdivision of the half-grid $H$  in any thick-ended graph \G, but allows the metric of $H$ to be arbitrarily distorted by that of \G. This is unsatisfactory in contexts where the geometry matters, e.g.\ in geometric group theory. In this section we seek mild conditions under which we can improve on Halin's theorem in order to preserve some of the geometry of $H$.
\medskip

Let $H,G$ be graphs, and let $f\colon H \to G$ be a subdivision of $H$ in $G$.
As above, we say that $f$ is \defi{diverging}, if \fe\ two sequences \seq{x},\seq{y}\ of points of $H$ \st\ $d_H(x_i,y_i) \to \infty$, we have $d_G(f(x_i),f(y_i)) \to \infty$.  For example, if $f$ is a quasi-isometry, or more generally a coarse embedding, then $f$ is diverging. Notice however that this does not mean that if $H$ is a graph coarsely embeddable in another graph \G, then \g contains a diverging subdivision of $H$; for example, there is no diverging subdivision of the square grid into $\mathbb{H}$, because the latter has no vertex of degree four. Our next result says that specifically for $\mathbb{H}$,  a coarse embedding implies a diverging subdivision.

\begin{proposition} \label{qi preserve div}
The property of containing a diverging subdivision of~$\mathbb{H}$ is preserved by coarse embeddings between bounded-degree graphs.
\end{proposition}

In particular, if $G=Cay(\Gamma, S)$ is a finitely generated \Cg, then its containing a diverging subdivision of $\mathbb{H}$ is independent of the choice of the generating set $S$. 

\begin{proof}
Let $\varphi\colon V(G)\to V(H)$ be a coarse embedding between two graphs of bounded degree and let $f\colon \mathbb{H}\to G$ be a diverging subdivision of~$\mathbb{H}$ in~$G$.
Then there are $L\geq 0$ and $K\geq 0$ such that $d(\varphi(x),\varphi(y))\leq L$ for all adjacent $x,y\in V(G)$ and $d(x,y)\leq K$ for all $x,y\in V(G)$ with $d(\varphi(x),\varphi(y))\leq 2L$.
Since $f$ is diverging, there exists $M>0$ such that $d(f(x),f(y))>K$ for all points $x$ and $y$ in~$\mathbb H$ with $d(x,y)>M$.
Let $g\colon \mathbb{H}\to f(\mathbb{H})$ be another subdivision of~$\mathbb{H}$ in~$G$ such that $d(g(x),g(y))>2K$ for all $x,y\in V(\mathbb{H})$.
Then $g$ is diverging, since $f$ is.
We claim that $\varphi(g(\mathbb{H}))$ contains a minor of~$\mathbb{H}$ in~$H$ and thus $H$ contains $\mathbb{H}$ as a subdivision.

For all adjacent $x,y\in V(g(\mathbb{H}))$, let $P_{xy}$ be a shortest $\varphi(x)$-$\varphi(y)$ path in~$H$ with $P_{xy}=P_{yx}$.
Its length is at most~$L$.
Let $H'$ be the subgraph of~$H$ induced by $\varphi(g(\mathbb{H}))$ and all paths $P_{xy}$.
For adjacent $u,v\in V(\mathbb{H})$, let $Q_{uv}$ be the $\varphi(g(u))$-$\varphi(g(v))$ walk in~$H'$ that is induced by the paths $P_{xy}$ for adjacent vertices $x,y$ on the $g$-image of $uv$ in~$G$.
If distinct $Q_{uv}$ and $Q_{u'v'}$ share a vertex, there are paths $P_{xy}$ and $P_{x'y'}$ on $Q_{uv}$ and $Q_{u'v'}$, respectively, that have a common vertex.
Thus, we have $d(\varphi(x),\varphi(x'))\leq 2L$ and hence $d(x,x')\leq K$ by the choice of~$K$.
By definition of~$M$ and~$g$, the sets $\{u,v\}$ and $\{u',v'\}$ have a common vertex.

Let $\{U,V\}$ be the canonical bipartition of~$\mathbb{H}$.
For every $u\in U$, set
\begin{align*}
V_u:=\bigcup\{&V(P_{xy})\mid v\in N(u), x,y\in g(uv), \text{ and }\\&d(g(u),g(x))\leq K\text{ and }d(g(u),g(y))\leq K\}.
\end{align*}
For adjacent $u\in U$ and $v\in V$, we denote by $Q'_{vu}$ the subwalk of~$Q_{uv}$ with $\varphi(g(v))$ as one end vertex and the vertex before the first vertex of~$V_u$ as its other end vertex.
For $v\in V$, set
\[
V_v:=\bigcup\{V(Q'_{vu})\mid u\in N(v)\}.
\]
Our arguments above show that the elements of $\{V_u,V_v\mid u\in U, v\in V\}$ are pairwise disjoint.
Thus, they form branch sets of a minor of~$\mathbb{H}$ in~$H$.

Since compositions of diverging functions are diverging, we deduce that the subdivision of~$\mathbb{H}$ we just constructed is diverging.
\end{proof}

We cannot drop the condition that the subdivision of $\mathbb{H}$ be diverging in \Prr{qi preserve div}. To see this, let $H'$ denote a half-lattice, e.g.\ the intersection of $\mathbb{H}$ with the half-plane $\{x,y\in \R^2 \mid y\geq 0\}$. Let $H''$ be a copy of $H'$, and join each pair of corresponding vertices of $H'$ and $H''$ by an edge to define a new graph~$G$. Then \G\ is quasi-isometric to $H'$, and contains a full lattice while $H'$ does not.

\begin{rem}\label{rem planar H diverging}
Let \g be a plane graph.
Since $\mathbb H$ is 3-connected, it has a unique embedding into the plane by Whitney's theorem~\cite{whitney_congruent_1932}.
It follows easily from this that every subdivision of~$\mathbb H$ in~\g is diverging.
\end{rem}

An 1-ended \lf\ graph can fail to contain a diverging subdivision of $\mathbb{H}$ even if it has a $K_\infty$ minor, as shown by the above example of a  `two-storey' half-grid. This, combined with Remark~\ref{rem planar H diverging},  motivates
\begin{problem} \label{prob diverging H}
Let \g be a (non-planar) \lf, 1-ended, \vt\ (or \qt) graph. Must \g contain a diverging subdivision of $\mathbb{H}$?
\end{problem}

We cannot ask for a coarse embedding of $\mathbb{H}$ here instead of a diverging one: a coarse embedding of $\mathbb{H}$ in \g implies that \g has asymptotic dimension at least 2, but the lamplighter graph over $\Z$ has asymptotic dimension 1 \cite{BrDyLaAss}. This suggests that diverging subdivisions might be just the right notion if one wants a geometric variant of Halin's theorem for groups.

\medskip
Perhaps we can drop the vertex-transitivity assumption in \Prb{prob diverging H} if we are happy with a diverging subdivision of half of $\mathbb{H}$:

\begin{problem} \label{prob diverging H bd}
Let \g be an 1-ended graph with a thick end and bounded degrees. Must \g contain a diverging subdivision of the half-grid?
\end{problem}

A positive answer to \Prb{prob diverging H} or \Prb{prob diverging H bd} would immediately yield a positive answer to
\begin{conjecture} \label{conj rays}
Let \g be a bounded-degree (vertex-transitive) graph. Suppose \g contains an infinite family of pairwise disjoint rays. Then \g contains an infinite family of pairwise diverging rays.
\end{conjecture}

Recall that \Lr{diverging rays} yields just two diverging rays in such a \G. Its proof method relies on a geometric version of Menger's theorem, proved for pairs of paths in \cite{AHJKW,GeoPapMin}. It was conjectured that the statement generalises to any number of paths, and this conjecture, combined with the proof of \Lr{diverging rays} would imply \Cnr{conj rays}. However, Nguyen, Scott \& Seymour~\cite{NSS-CounterExCoarseMenger} recently gave a counterexample to said conjecture.

For \nin, let $R_n$ be the graph consisting of $n$ rays with a common starting vertex and no other common vertices. 
If \g is vertex-transitive and 1-ended, then it is not hard to prove that \ti\ a quasi-isometrically embedded copy of $R_3$ in \G. In particular, \ti\ a triple of pairwise diverging rays. Panos Papasoglu (private communication) has asked the following: 

\begin{question} \label{Q pap}
Let \g be an 1-ended vertex-transitive graph. Must \g contain a quasi-isometrically embedded copy of $R_4$?
\end{question}

We do not even know how to find a quadruple of pairwise diverging rays in such \G.

Recall that in the proof of \Lr{full H planar} we were able to exploit planarity to construct an infinite family of pairwise disjoint rays. This motivates the following problem.

\begin{problem} \label{prob codegree}
Let \g be an 1-ended, finitely presented, \Cg. Must \g have a planar subgraph $H$ with bounded co-degree and a thick end? Can we choose this $H$ to be coarsely embedded in \G?
\end{problem}
The same question can be asked for a large-scale-simply-connected\footnote{We say that \g is  \defi{large-scale-simply-connected} if its fundamental group is generated by cycles of bounded length.} graph \g with bounded degree and a thick end.

\section{Well-quasi-ordering \Cg s} \label{sec wqo}

A well-known conjecture of Thomas \cite{ThoWel} postulates that the countable graphs are well-quasi-ordered under the minor relation. The analogous statement for finite graphs is the celebrated Graph Minor Theorem of Robertson \& Seymour \cite{GMXX}. Thomas' conjecture may well be false in this generality, and in any case it is very difficult. But the following may be within reach:

\begin{conjecture} \label{conj wqo}
The countable Cayley graphs are well-quasi-ordered under the minor relation. 
\end{conjecture}

Several results are known that make progress towards this. Thomassen proved that every \lf, 1-ended, non-planar, vertex-transitive graph has the infinite clique $K_\infty$ as a minor \cite{Thomassen-Hadwiger}. On the other hand, it is known that every planar 1-ended graph admits an embedding into $\R^2$ without accumulation points of vertices \cite[Lemma 12]{ThomassenRichter}, and that every such graph is a minor of $\mathbb{H}$ \cite{KuhMin}. Combining these results with our  \Cr{full Halin VT} settles the restriction of \Cnr{conj wqo} to 1-ended graphs in a strong way. Call two graphs $G,H$ \defi{minor-twins}, if both $G<H$ and $H<G$ hold, and notice that this is an equivalence relation. The above facts combined yield

\begin{proposition} \label{prop just two}
There are exactly two minor-twin classes of \lf, 1-ended, vertex-transitive graphs.\footnote{Namely, that of $\mathbb{H}$, and that of e.g.\ the cubic grid $\Z^3$.}
\end{proposition}

It remains to consider the multi-ended case of \Cnr{conj wqo}. It is easy to come up with examples of arbitrarily long strictly decreasing chains $G_1 > G_2 > \ldots G_k$ of 2-ended \Cg s: let $G_i$ be the cartesian product $C_{k-i+3} \times R$, where $R$ denotes the 2-way infinite path, and $C_n$ denotes the cycle. (Thus $G_i$ is the standard \Cg\ of the abelian product of one finite and one infinite cyclic group, where the order of the former decreases with $i$.)

Recall that the \defi{degree} of an end $\omega$ of a graph \g is the maximum number of disjoint rays of \g that belong to $\omega$. It is well-known that if \g is a 2-ended \Cg, then its two ends have the same degree and that degree is finite; this follows from the proof of \cite[Theorem 7]{Diestel1993}, see also \cite{MirRue2022}. We expect that \fe\ $n\in \N$, there are finitely many, perhaps even uniformly boundedly many,  minor-twin classes of 2-ended \Cg s. 

Esperet, Giocanti \& Legrand-Duchesne \cite{EGL-QuasiTransitiveGraphsAvoidingMinor} recently proved that a \Cg\ with arbitrarily large finite clique minors must have a $K_\infty$ minor, and as a result, every inaccessible \Cg\ has a $K_\infty$ minor. Thus we can restrict our attention accessible graphs in \Cnr{conj wqo}. 

Recall that  every \qt\ \lf\ graph without any thick end is \qi\ to a tree \cite[Theorem 5.5]{KroMolQua}. In this case we can try to use, or generalise, Nash-Williams' theorem that the infinite trees are well-quasi-ordered (in fact better-quasi-ordered) under the (topological) minor relation \cite{NWbqo}. For graphs with a thick end \Tr{full Halin VT general}, combined with Thomassen's aforementioned result, might again be useful.

\subsection{Finer minor notions} \label{sec}

Even if  \Cnr{conj wqo} is true, it seems too weak to have any group-theoretic consequences. In this section we pose stronger versions that might be more consequential.

The idea is to replace the minor relation $<$ by a finer one. An obvious candidate is the \defi{shallow minor} relation $\shm$, defined just like $<$, except that we restrict the branch sets to have uniformly bounded size. 

But since we are discussing groups, notions that are stable under changing the generating set of a \Cg\ are more likely to be fruitful. Therefore, it is natural to consider the coarse embeddability relation $\coem$, or the diverging subdivision  relation $\divm$. One could also consider the diverging embedding relation, i.e.\ existence of a diverging map between two graphs as defined in \Sr{embs}. Note that all these relations are transitive, hence they define quasi-orders on the set of (countable, Cayley) graphs.

\begin{question} \label{Q wqo}
For which of the aforementioned relations is the class of countable Cayley graphs well-quasi-ordered? 
\end{question}

This question is interesting enough when restricted to the  \lf\ 1-ended case. Robert Kropholler (private communication) has found an infinite $\coem$-antichain even within the class of finitely presented \Cg s adapting ideas from \cite{BriBraGap, KroPenCoa}. Thus Question~\ref{Q wqo} seems particularly interesting for the relation $\divm$.

\end{document}